\documentclass[preprint]{elsarticle}

\usepackage[numbers]{natbib}
\usepackage{breqn}
\usepackage[hidelinks]{hyperref}
\usepackage{amsmath}
\usepackage{amssymb}
\usepackage{amsthm}
\usepackage{tikz}
\usepackage{caption}

\newtheorem{theorem}{Theorem}
\newtheorem{claim}{Claim}[theorem]
\newtheorem{proposition}[theorem]{Proposition}
\newtheorem{lemma}[theorem]{Lemma}

\def\claimqed{\smash{\scalebox{.75}[0.75]{$(\square$)}}}

\title{Characterisations for split graphs and unbalanced split graphs\tnoteref{t1}}
\tnotetext[t1]{This document is the results of the research
project funded by the European Social Fund (ESF).}
\author{Hany Ibrahim}
\address{University of applied science Mittweida, Germany}
\date{}

\begin{document}

\begin{abstract}
	We introduce a characterisation for split graphs by using edge contraction. First, we use it to prove that any ($2K_{2}$, claw)-free graph with $\alpha(G) \geq 3$ is a split graph. Next, we apply it to characterise any pseudo-split graph. Finally, by using edge contraction again, we characterise unbalanced split graphs, which we use to characterise Nordhaus--Gaddum graphs.
\end{abstract}

\maketitle

\section{Introduction}
Given graph $G$, we denote its vertex set by $V(G)$ ($V$ in short) if the underlying graph is clear. Similarly, the edge set of $G$ is $E(G)$ ($E$ in short). For vertex $v \in V$, the set of vertices adjacent to $v$ in $G$ is denoted by $N(v)$. Further, for a set of vertices, $S \subseteq V$, the set of vertices adjacent to $v$ in $S$ is denoted by $N_{S}(v)$ and defined as $N_{S}(v) = N(v) \cap S$. The neighbour of $S$, denoted by $N(S)$, is $\bigcup_{v \in S}N(v) \setminus S$. A vertex set, $S$, is called \emph{dominating} if $N(S) \cup S = V$. We denote a cycle graph, a complete graph, and an edgeless graph by $C_{n}$, $K_{n}$, and $E_{n}$, respectively, where $n$ is the order of the graph. 
Graph $G[S]$ has $S$ as the vertex set, and two vertices are adjacent if and only if they are adjacent in $G$. $G[S]$ is called the \emph{vertex-induced subgraph} of $G$. A \emph{clique} is a vertex set that induces a complete subgraph in $G$. A set of mutually nonadjacent vertices is \emph{independent}. Furthermore, if two edges have no common vertex, they are called \emph{independent}. We denote a graph with four vertices and two independent edges by $2K_{2}$.

Graph $G$ is called \emph{$H$-free} if every vertex-induced (induced in short) subgraph of $G$ is not isomorphic to graph $H$. Furthermore, $G$ is \emph{$(H_{1},H_{2},\dots,H_{k})$-free} if every vertex-induced subgraph of $G$ is not isomorphic to any graph $H_{i}$, where $1 \leq i \leq k$. For two disjoint graphs, $G$ and $H$, the graph constructed from $G \cup H$ by adding edges from any vertex in $G$ to any vertex in $H$ can be denoted by $G \vee H$. 

By \emph{identifying} two adjacent vertices, $u$ and $v$, or \emph{contracting} the edge between them, we obtain a graph constructed from $G$ by adding an edge between $u$ and every neighbour of $v$. Then, $v$, along with all loops, is deleted, followed by the deletion of all but one edge that forms multiple edges between any two vertices. We denote the graph obtained by identifying $u$ and $v$ as $G/uv$. If $e$ is the edge between $u$ and $v$, then we denote graph $G/uv$ by $G/e$. 
In this study, we assume that any graph is a connected simple graph, unless otherwise stated. For any other notion, we follow \cite{bondy2000graph}. Long proofs are divided into small claims in which only the ambiguous aspects are proven, for example, Lemmas \ref{C4 characterisation}, Lemma \ref{2K2 characterisation} and Theorem \ref{Proposition: Proving 2K2,claw by using split graph edge contraction characterisation}.

The $KS$-partition of a graph is a partition of the vertex set, where $K$ is a clique and $S$ is an independent set. Graph $G$ is called \emph{split} if it admits a $KS$-partition. Split graphs were introduced in \cite{hammer1977} and characterised as follows:
\begin{theorem}\cite{hammer1977}\label{Theorem: split graph hammer characterisation}
	Graph $G$ is split if and only if $G$ is ($2K_{2}$, $C_{5}$, $C_{4}$)-free.
\end{theorem}
In addition, split graphs were characterised in \cite{hammer1977} as chordal graphs whose complements are also chordal. Furthermore, these graphs were characterised by their degree sequences in \cite{hammer1981splittance}. 

In Section \ref{Section: split characterisation}, we present the characterisation of the split graphs. Then, in Section \ref{Section: 2K2,claw free}, we use it to prove that any ($2K_{2}$, claw)-free graph with $\alpha(G) \geq 3$ is a split graph. Finally, in Section \ref{Section: unbalanced split}, we characterise unbalanced split graphs by using edge contraction, and we use this to characterise Nordhaus--Gaddum graphs.

\section{Split graphs characterisation}\label{Section: split characterisation}
\begin{proposition}\label{Proposition: contraction of vertex adjacent to some neighbours}
	Let $G$ be a graph, $C \subseteq V(G)$ with $u \in C$ and $v \notin C$, where $u$ and $v$ are adjacent. If $N_{C}(v)\setminus \{u\} \subseteq N_{C}(u)$, then $G[C]$ is isomorphic to $G/uv[C]$.
\end{proposition}

\begin{proposition}\label{Proposition: contraction of vertices adjacent outside C}
	Let $G$ be a graph and $C \subseteq V(G)$. If $u$ and $v \notin C$, where $u$ and $v$ are adjacent, then $G[C]$ is isomorphic to $G/uv[C]$.
\end{proposition}

\begin{proposition}\label{Proposition: contraction of vertices not dominated by C}
	Let $G$ be a graph and $C \subseteq V(G)$. If $C$ is not dominant, then there is an edge, $e \in E(G)$, such that $G[C]$ is isomorphic to $G/e[C]$.
\end{proposition}
\begin{proof}
	If $C$ is not dominant, then there is a vertex, $v \notin C$, that is not adjacent to any vertex in $C$. Because $G$ is connected, there is a vertex, $u \notin C$, such that $u$ is adjacent to $v$. Hence, according to Proposition \ref{Proposition: contraction of vertices adjacent outside C}, there exists an edge, $e \in E(G)$, such that $G[C]$ is isomorphic to $G/e[C]$.
\end{proof}

\begin{proposition} \label{Lemma: contraction of a cycle}
	If $G$ is a graph that is isomorphic to $C_{n}$ with $n \geq 4$, then $G/e$ is isomorphic to $C_{n-1}$ for any edge, $e\in E$.
\end{proposition}
\begin{proposition} \label{Lemma: contraction of a complete graph}
	If $G$ is a graph that is isomorphic to $K_{n}$ with $n \geq 2$, then $G/e$ is isomorphic to $K_{n-1}$ for any edge, $e\in E$.
\end{proposition}

Before we characterise the split graphs, we present a list of special graphs in Figure \ref{Figure: list of forbidden graphs}. These graphs have an interesting property; although they are not split graphs, we can construct one if we contract an edge in any of them. We prove that these are the only graphs that possess such a property. 
\begin{figure}[ht!]
	\centering
	\begin{tikzpicture}[hhh/.style={draw=black,circle,inner sep=2pt,minimum size=0.2cm}]

	\begin{scope}[shift={(0,0)}]
	\node[hhh] 	(a) at (45:1cm) 	{};
	\node[hhh]  (b) at (135:1cm) 	{};
	\node[hhh] 	(c) at (225:1cm) 	{};
	\node[hhh] 	(d) at (-45:1cm) 	{};
	
	\node[hhh] 	(g) at (0:0cm) 	{};
	\node[align=center] 	(h) at (-90:2.2cm) 	{$H_{2}$\\$(W_{4})$};

	\draw (a) -- (b) -- (c) -- (d) --(a);
	\draw (a) -- (g) -- (b)  (c) -- (g) -- (d) ;
	\end{scope}

		\begin{scope}[shift={(4,0)}]
		\node[hhh] 	(a) at (45:1cm) 	{};
		\node[hhh]  (b) at (135:1cm) 	{};
		\node[hhh] 	(c) at (225:1cm) 	{};
		\node[hhh] 	(d) at (-45:1cm) 	{};
		
		\node[hhh] 	(f) at (90:1.2cm) 	{};
		\node[hhh] 	(g) at (-90:1.2cm) 	{};
		\node[align=center] 	(h) at (-90:2.2cm) 	{$H_{3}$ \\ $(E_{2} \vee C_{4})$};

		\draw (a) -- (b) -- (c) -- (d) --(a);
		\draw (a) -- (g) -- (b)  (c) -- (g) -- (d) (a) -- (f) -- (b)  (c) -- (f) -- (d);
	\end{scope}

	\begin{scope}[shift={(-4,0)}]
	\node[hhh] 	(a) at (0:0.8cm) 	{};
	\node[hhh]  (b) at (90:1.2cm) 	{};
	\node[hhh] 	(c) at (180:0.8cm) 	{};
	\node[hhh] 	(d) at (-90:1.2cm) 	{};
	
	\node 	(g) at (0:0cm) 	{$\dots$};
	\node[align=center] 	(h) at (-90:2.1cm) 	{$H_{1}^{l}$ \\  $(K_{2,l}, l \geq 2)$};
	
	\draw (a) -- (b) -- (c) -- (d) --(a);
	\end{scope}

	\begin{scope}[shift={(-5,-4.2)}]
	\node[hhh] 	(a) at (45:1cm) 	{};
	\node[hhh]  (b) at (135:1cm) 	{};
	\node[hhh] 	(c) at (225:1cm) 	{};
	\node[hhh] 	(d) at (-45:1cm) 	{};

	\node[align=center]  	(h) at (-90:1.7cm) 	{$H_{4}$ \\ $(2K_{2})$};
	
	\draw (a) -- (b)  (c) -- (d) ;
	\end{scope}
	
	\begin{scope}[shift={(-1.8,-4.2)}]
		\node[hhh] 	(a) at (0:1.2cm) 	{};
		\node[hhh]  (b) at (0:0.6cm) 	{};
		\node[hhh] 	(c) at (180:0.6cm) 	{};
		\node[hhh] 	(d) at (180:1.2cm) 	{};
		\node[hhh] 	(e) at (0:0cm) 	{};
		
		\node[align=center]  	(h) at (-90:1.7cm) 	{$H_{5}$ \\ $(P_{5})$};
		
		\draw (a) -- (b) --(e) -- (c) -- (d);
	\end{scope}

	\begin{scope}[shift={(1.8,-4.2)}]
	\node[hhh] 	(a) at (90:1cm) 	{};
	\node[hhh]  (b) at (90:0.4cm) 	{};
	\node[hhh] 	(c) at (225:1.2cm) 	{};
	\node[hhh] 	(d) at (-45:1.2cm) 	{};
	\node[hhh] 	(e) at (-90:0.3cm) 	{};

	\node[align=center]  	(h) at (-90:1.7cm) 	{$H_{6}$ \\ $(hammer)$};
	
	\draw (a) -- (b) --(e) -- (c) -- (d) (d) -- (e);
	\end{scope}

	\begin{scope}[shift={(5,-4.2)}]
	\node[hhh] 	(a) at (45:1.2cm) 	{};
	\node[hhh]  (b) at (135:1.2cm) 	{};
	\node[hhh] 	(c) at (225:1.2cm) 	{};
	\node[hhh] 	(d) at (-45:1.2cm) 	{};
	\node[hhh] 	(e) at (180:0cm) 	{};

	\node [align=center] 	(h) at (-90:1.7cm) 	{$H_{7}$ \\ $(butterfly)$};
	
	\draw (a) -- (b) --(e) -- (c) -- (d) (d) -- (e) -- (a);
	\end{scope}

	\end{tikzpicture}	
	\caption{}
	\label{Figure: list of forbidden graphs}
\end{figure}

\begin{lemma}\label{C4 characterisation}
	Given graph $G$, if $G$ has an induced $C_{4}$, then either $G$ is isomorphic to one of $\{H_{1}^{l}: l \geq 2\}$ , $H_{2}$, and $H_{3}$ or there is an edge, $e \in E(G)$, such that $G/e$ has an induced $C_{4}$.
\end{lemma}
\begin{proof}
	\renewcommand{\qedsymbol}{\claimqed}
	Let $C=\{p,q,r,s\} \subseteq V(G)$ that induces $C_{4}$ with edges $pq,qr,rs$, and $sp$. 
	
	Based on Proposition \ref{Proposition: contraction of vertices adjacent outside C}, we observe the following claim.
	\begin{claim}\label{Claim1: C4 characterisation}
		If there are two adjacent vertices not in $C$, then there is an edge, $e \in E(G)$, such that $G/e$ has an induced $C_{4}$.
	\end{claim}

	By applying Proposition \ref{Proposition: contraction of vertices not dominated by C}, the following claim is achieved.
	\begin{claim}\label{Claim2: C4 characterisation}
		If $C$ is not dominant, then there is an edge, $e \in E(G)$, such that $G/e$ has an induced $C_{4}$.
	\end{claim}

	By using Proposition \ref{Proposition: contraction of vertex adjacent to some neighbours}, we obtain the following claim.
	\begin{claim}\label{Claim3: C4 characterisation}
		If $v \notin C$ such that $|N_{C}(v)|=1$ or $v$ is adjacent to exactly two adjacent vertices in $C$ or $|N_{C}(v)|=3$, then there is an edge, $e \in E(G)$, such that $G/e$ has an induced $C_{4}$.
	\end{claim}

	\begin{claim}\label{Claim4: C4 characterisation}
		If for every $v \notin C$, $v$ is adjacent to exactly the same two nonadjacent vertices in $C$, then $G$ is isomorphic to $H_{1}^{l}$, where $l =|V(G) \setminus C| + 2$.
	\end{claim}

	\begin{claim}\label{Claim5: C4 characterisation}
		If $|V(G) \setminus C|=1$, and the vertex not in $C$ is adjacent to all vertices in $C$, then $G$ is isomorphic to $H_{2}$.
	\end{claim}
	\begin{claim}\label{Claim6: C4 characterisation}
		If $|V(G) \setminus C|=2$, and each of the two vertices not in $C$ is adjacent to all vertices in $C$, then $G$ is isomorphic to $H_{3}$.
	\end{claim}	
	
	\begin{claim}\label{Claim8: C4 characterisation}
		If $|V(G) \setminus C|=3$ and each of the three vertices in $V(G) \setminus C$ is adjacent to all vertices in $C$, then there is an edge, $e \in E(G)$, such that $G/e$ has an induced $C_{4}$.
	\end{claim}
	\begin{proof}
		If $ \{u,v,w\} \subseteq V(G) \setminus C $ such that $|N_{C}(u)|=|N_{C}(v)|=|N_{C}(w)|=4$, then $\{p,r,v,w\}$ induces $C_{4}$ in $G/uq$.
	\end{proof}

	\begin{claim}\label{Claim9: C4 characterisation}
		If there are two vertices, $u$ and $v \notin C$, where $|N_{C}(u)|=|N_{C}(v)|=2$ and $N_{C}(u) \cap N_{C}(v) = \{\}$, then there is an edge, $e \in E(G)$, such that $G/e$ has an induced $C_{4}$.
	\end{claim}
	\begin{proof}
		If $v$ is adjacent to exactly two adjacent vertices in $C$, then according to Claim \ref{Claim3: C4 characterisation}, there is an $e \in E(G)$ such that $G/e$ has an induced $C_{4}$. Therefore, we assume that $G$ has no vertex that is adjacent to exactly two adjacent vertices in $C$. 
		Let $v \notin C$ be adjacent to exactly two nonadjacent vertices in $C$, say $p,r$. If there is a vertex, $u \notin C$, that is adjacent to two different nonadjacent vertices in $C$, say $q,s$, then the vertex set, $\{p,r,s,v\}$, induces $C_{4}$ in $G/uq$.
	\end{proof}
	
	The following claim is proved in a similar way to the proof of Claim \ref{Claim9: C4 characterisation}.
	\begin{claim}\label{Claim10: C4 characterisation}
		If there are two vertices, $u$ and $v \notin C$, where $N_{C}(v)=2$ and $N_{C}(u)=4$, then there is an edge, $e \in E(G)$, such that $G/e$ has an induced $C_{4}$.
	\end{claim}
\renewcommand{\qedsymbol}{$\square$}
	Based on Claims \ref{Claim1: C4 characterisation}, $\dots$, \ref{Claim10: C4 characterisation}, the proof is complete.
\end{proof}

\begin{lemma}\label{2K2 characterisation}
		Given graph $G$, if $G$ has an induced $2K_{2}$, then either $G$ is isomorphic to one of $2K_{2}$, $P_{5}$, Hammer, Butterfly, and $C_{6}$ or there is an edge, $e \in E(G)$, such that $G/e$ has an induced $2K_{2}$ or $C_{4}$.
\end{lemma}
\begin{proof}
		\renewcommand{\qedsymbol}{\claimqed}
	Let $C=\{p,q,r,s\} \subseteq V$, where $pq,rs \in E$, and $C$ induces $2K_{2}$ in $G$.

Based on Proposition \ref{Proposition: contraction of vertices adjacent outside C}, we observe the following claim.
\begin{claim}\label{Claim1: 2K2 characterisation}
	If there are two adjacent vertices not in $C$, then there is an edge, $e \in E(G)$, such that $G/e$ has an induced $2K_{2}$.
\end{claim}

By applying Proposition \ref{Proposition: contraction of vertices not dominated by C}, the following claim is achieved.
\begin{claim}\label{Claim2: 2K2 characterisation}
	If $C$ is not dominant, then there is an edge, $e \in E(G)$, such that $G/e$ has an induced $2K_{2}$.
\end{claim}

	By using Proposition \ref{Proposition: contraction of vertex adjacent to some neighbours}, we obtain the following claim.
\begin{claim}\label{Claim3: 2K2 characterisation}
	If $v \notin C$, such that $N_{C}(v)=1$ or $v$ is adjacent to exactly two adjacent vertices in $C$, then there is an edge, $e \in E(G)$, such that $G/e$ has an induced $2K_{2}$.
\end{claim}

\begin{claim}\label{Claim4: 2K2 characterisation}
	If $V = C$, then $G$ is isomorphic to $2K_{2}$.
\end{claim}
\begin{claim}\label{Claim5: 2K2 characterisation}
	If $|V \setminus C|=1$ and the vertex in $V \setminus C$ is adjacent to exactly two nonadjacent vertices in $C$, then $G$ is isomorphic to $P_{5}$.
\end{claim}
\begin{claim}\label{Claim6: 2K2 characterisation}
	If $|V \setminus C|=1$ and the vertex not in $C$ is adjacent to exactly three vertices in $C$, then $G$ is isomorphic to Hammer.
\end{claim}
\begin{claim}\label{Claim7: 2K2 characterisation}
	If $|V \setminus C|=1$ and the vertex not in $C$ is adjacent to all vertices in $C$, then $G$ is isomorphic to Butterfly.
\end{claim}

\begin{claim}\label{Claim8: 2K2 characterisation}
	If there are two vertices, $u$ and $v \notin C$, such that $|N_{C}(u)|=|N_{C}(v)|=2$, then $G$ is isomorphic to $C_{6}$ or there is an edge, $e \in E(G)$, such that $G/e$ has an induced $2K_{2}$ or $C_{4}$.
\end{claim}
\begin{proof}
	If $u$ or $v$ is adjacent to two adjacent vertices in $C$, then according to Claim \ref{Claim3: 2K2 characterisation}, there is an edge, $e \in E(G)$, such that $G/e$ has an induced $2K_{2}$. Therefore, we assume that neither $u$ nor $v$ is adjacent to two adjacent vertices in $C$.
	If $N_{C}(u)=N_{C}(v)$, say $N_{C}(u)=\{p,r\}$, then vertex set $\{p,r,u,v\}$ induces $C_{4}$ in $G$. According to Lemma \ref{C4 characterisation}, there is an edge, $e \in E(G)$, such that $G/e$ has an induced $C_{4}$. If $|N_{C}(u) \cap N_{C}(v)| = 1$, then there is a vertex set in $G$ that induces $C_{5}$. Thus, according to Proposition \ref{Lemma: contraction of a cycle}, there is an edge, $e \in E(G)$, such that $G/e$ has an induced $C_{4}$.
	If $N_{C}(u) \cap N_{C}(v) = \{\}$, say $N_{C}(u)=\{p,r\}$ and $N_{C}(v)=\{q,s\}$, then vertex set $\{p,q,r,s,u,v\}$ induces $C_{6}$ in $G$. If $w \notin C$, where $|N_{C}(w)|$ equals either $3$ or $4$, then $G$ has an induced $C_{4}$. Thus, $G$ is isomorphic to $C_{6}$.
\end{proof}

\begin{claim}\label{Claim9: 2K2 characterisation}
	If there are two vertices, $u$ and $v \notin C$, such that $|N_{C}(u)|=|N_{C}(v)|=3$, then there is an edge, $e \in E(G)$, such that $G/e$ has an induced $C_{4}$.
\end{claim}
\begin{proof}
	For any two vertices, $u$ and $v \notin C$, such that $|N_{C}(u)|=|N_{C}(v)|=3$, $G$ has an induced $C_{4}$ or $C_{5}$. Thus, according to Lemma \ref{C4 characterisation} and Proposition \ref{Lemma: contraction of a cycle}, there is an edge, $e \in E(G)$, such that $G/e$ has an induced $C_{4}$.
\end{proof}

\begin{claim}\label{Claim10: 2K2 characterisation}
	If there are two vertices, $u$ and $v \notin C$, such that $|N_{C}(u)|=|N_{C}(v)|=4$, then there is an edge, $e \in E(G)$, such that $G/e$ has an induced $C_{4}$.
\end{claim}
\begin{proof}
	For any two vertices, $u$ and $v \notin C$, such that $|N_{C}(u)|=|N_{C}(v)|=4$, $G$ has an induced $C_{4}$. Thus, according to Lemma \ref{C4 characterisation}, there is an edge, $e \in E(G)$, such that $G/e$ has an induced $C_{4}$.
\end{proof}

In a similar way to the proof of Claim \ref{Claim9: 2K2 characterisation}, we can prove the following claim.
\begin{claim}\label{Claim11: 2K2 characterisation}
	If there are two vertices, $u$ and $v \notin C$, such that $|N_{C}(u)|=2$ and $|N_{C}(v)|=3$, then there is an edge, $e \in E(G)$, such that $G/e$ has an induced $C_{4}$.
\end{claim}

In a similar way to the proof of Claim \ref{Claim10: 2K2 characterisation}, we can prove the following two claims.
\begin{claim}\label{Claim12: 2K2 characterisation}
	If there are two vertices, $u$ and $v \notin C$, such that $|N_{C}(u)|=2$ and $|N_{C}(v)|=4$, then there is an edge, $e \in E(G)$, such that $G/e$ has an induced $C_{4}$.
\end{claim}

\begin{claim}\label{Claim13: 2K2 characterisation}
	If there are two vertices, $u$ and $v \notin C$, such that $|N_{C}(u)|=3$ and $|N_{C}(v)|=4$, then there is an edge, $e \in E(G)$, such that $G/e$ has an induced $C_{4}$.
\end{claim}

\renewcommand{\qedsymbol}{$\square$}
Based on Claims \ref{Claim1: 2K2 characterisation}, $\dots$, \ref{Claim13: 2K2 characterisation}, the proof is complete.
\end{proof}

\begin{theorem}\label{Split graph characterisation}
	Given a connected graph, $G$, that is not isomorphic to any graph in Figure \ref{Figure: list of forbidden graphs}, $G$ is split if and only if $G/e$ is split for any $e \in E(G)$.
\end{theorem}
\begin{proof}
	Let $G$ be a split graph with a $KS$-partition for $V$. An edge in $E(G)$ is either connecting two vertices in $K$ or a vertex each in $K$ and $S$. Let $u,v \in K$. According to Proposition \ref{Lemma: contraction of a complete graph}, the order of the complete subgraph induced by $K$ in $G$ decreases by $1$ in $G/uv$. Hence, we can partition $V(G/uv)$ into $\{K \setminus \{v\},S\}$. Consequently, $G/uv$ is a split graph. Let $e$ be an edge in $E(G)$ between a vertex in $K$ and a vertex in $S$, say $v$. We can partition $V(G/e)$ into $\{K,S \setminus \{v\}\}$. Thus, $G/e$ is a split graph. 
	
	Conversely, we prove that if $G$ is not split, then $G/e$ is not split for at least an edge, $e \in E$. Consequently, and based on Theorem \ref{Theorem: split graph hammer characterisation}, we prove that if $G$ is not split, then $G/e$ has an induced $2K_{2}$, $C_{5}$, or $C_{4}$. 

	According to Proposition \ref{Lemma: contraction of a cycle}, if $G$ has an induced $C_{5}$, then there is an edge, $e \in E(G)$, such that $G/e$ is not split.

	According to Lemma \ref{C4 characterisation}, if $G$ has an induced $C_{4}$, then either $G$ is isomorphic to one of $\{H_{1}^{l}$: $l \geq 2\}$, $H_{2}$, and $H_{3}$ or there is an edge, $e \in E(G)$, such that $G/e$ is not split.

	According to Lemma \ref{2K2 characterisation}, if $G$ has an induced $2K_{2}$, then either $G$ is isomorphic to one of $2K_{2}$, $P_{5}$, Hammer, Butterfly or there is an edge, $e \in E(G)$, such that $G/e$ is not split.
	
\end{proof}

Theorem \ref{Split graph characterisation} shows that we can partition the set of all graphs into the following three parts:
\begin{itemize}
	\item The set of split graphs: the contraction of an edge in any graph in this set constructs a split graph.
	\item The set of graphs presented in Figure \ref{Figure: list of forbidden graphs}: these are non-split graphs in which the contraction of an edge in any graph constructs a split graph.
	\item The set of all non-split graph not presented in Figure \ref{Figure: list of forbidden graphs}: the contraction of at least one edge in any graph in this set constructs a non-split graph.
\end{itemize} 
This partition is presented in Figure \ref{Figure: Partition of graphs into split and not split}.
\begin{figure}[ht!]
	\centering
	\begin{tikzpicture}
		\begin{scope}[scale=0.8]
			
		\node[coordinate] (a1) at (1.5,2) {};
		\node[coordinate] (a2) at (1.5,-2) {};
		\node[coordinate] (a11) at (5,1.5) {};
		\node[coordinate] (a111) at (4,2) {};
		\node[coordinate] (a33) at (5.7,0) {};
		\node[coordinate] (a22) at (5,-1.5) {};
		\node[coordinate] (b1) at (-1.5,2) {};
		\node[coordinate] (b11) at (-5,1.5) {};
		\node[coordinate] (b33) at (-5.7,0) {};
		\node[coordinate] (b22) at (-5,-1.5) {};
		\node[coordinate] (b2) at (-1.5,-2) {};
		\draw[ultra thick]	(a1) --	(a2)  (b1) -- (b2) (a1) -- (b1) (a2) -- (b2);
		\draw[ultra thick] (a1) to[out=0,in=135] (a11) to[out=-45,in=45] (a22) to[out=225,in=0] (a2);
		\draw[ultra thick] (b1) to[out=180,in=45] (b11) to[out=225,in=135] (b22) to[out=-45,in=180] (b2);
		
		\node[text width=1.5cm,align=center] (split) at (3,0) {Split graphs}; 
		\node[text width=2cm,align=center] (figure1) at (0,0) {Graphs in Figure $1$}; 
		\node[text width=3cm,align=center] (Nonsplit) at (-3.5,0) {Non split graphs not in Figure $1$};
			
		\draw[->, ultra thick, blue] (a33).. controls +(down:0.1cm) and +(right:2 cm) .. (a22) node [pos=0.6,below, sloped] {$\forall_{e \in E} G/e$};
		\draw[->, ultra thick, blue] (0,2).. controls +(up:2cm) and +(up:2 cm) .. (a111) node [pos=0.5,above, sloped] {$\forall_{e \in E} G/e$};
		\draw[->, ultra thick, blue] (b33).. controls +(up:0.2cm) and +(left:2.5 cm) .. (b22) node [pos=0.6,below, sloped] {$\exists_{e \in E} G/e$};
		
	\end{scope}
	\end{tikzpicture}	
	\caption{}
	\label{Figure: Partition of graphs into split and not split}
\end{figure}

\section{($2K_{2}$, claw)-free graphs}\label{Section: 2K2,claw free}
 A graph is called a \emph{claw} if it has four vertices and three edges with one vertex adjacent to the other three. The \emph{clique number} of a graph is the maximum cardinality of a clique in the graph. The \emph{independent number} of graph $G$ is the maximum cardinality of an independent set in $G$ and is denoted by $\alpha(G)$. 
 
 Moreover, graph $G$ is called \emph{perfect} if the clique number of $H$ equals its chromatic number for any vertex-induced subgraph $ H$ of $ G$. A \emph{hole} is a vertex-induced cycle in $G$ with a length of at least four. An \emph{antihole} is a vertex-induced complement of a hole. In \cite{chudnovsky2006strong}, it was shown that
\begin{theorem}
	A graph is perfect if and only if it contains neither an odd hole nor an odd antihole.
\end{theorem}
In \cite{brause2019chromatic}, it was proved that
\begin{theorem}\label{Theorem: brause}
	If $G$ is connected ($2K_{2}$,claw)-free with $\alpha(G) \geq 3$, then $G$ is perfect.
\end{theorem}
We use the characterisation in Theorem \ref{Split graph characterisation} to prove the following:

\begin{theorem}\label{Proposition: Proving 2K2,claw by using split graph edge contraction characterisation}
	If $G$ is connected ($2K_{2}$,claw)-free with $\alpha(G) \geq 3$, then $G$ is split.
\end{theorem}
\begin{proof}
	\renewcommand{\qedsymbol}{\claimqed}
	For the sake of contradiction, we assume that $G$ is a connected ($2K_{2}$, claw)-free graph with $\alpha(G) \geq 3$, but it is not split. According to Theorem \ref{Split graph characterisation}, $G$ is either isomorphic to a graph in Figure \ref{Figure: list of forbidden graphs}: $H_{i}$, where $1 \leq i \leq 7$, or there is an edge, $e \in E$, where $G/e$ is not split. 
	
	\begin{claim}\label{Claim1: 2K2,claw-free characterisation}
		Graphs $H_{4},H_{5},H_{6}$, and $H_{7}$ are not $2K_{2}$-free.
	\end{claim}

	\begin{claim}\label{Claim2: 2K2,claw-free characterisation}
		Graph $H_{1}^{l}$, where $l \geq 2$, is not $claw$-free if $\alpha(H_{1}^{l}) \geq 3$.
	\end{claim}
	\begin{proof}
		The vertex set of $H_{1}^{l}$ is a union of set $C=\{p,q,r,s\}$ that induces $C_{4}$ and $l-2$ vertices that are adjacent to exactly two nonadjacent vertices in $C$, say $q,s$. If $\alpha(H_{1}^{l}) \geq 3$, then $l \geq 3$. Let $v \notin  C$. Vertex set $\{p,q,r,v\}$ induces a claw.
	\end{proof}

	\begin{claim}\label{Claim3: 2K2,claw-free characterisation}
		Each graph $H_{2}$ and $H_{3}$ has no independent set with a cardinality larger than $2$.
	\end{claim}

	Based on Claims \ref{Claim1: 2K2,claw-free characterisation}, \ref{Claim2: 2K2,claw-free characterisation}, and \ref{Claim3: 2K2,claw-free characterisation}, the following claim follows.
	\begin{claim}\label{Claim4: 2K2,claw-free characterisation}
		If $G$ is a connected ($2K_{2}$,claw)-free graph with $\alpha(G) \geq 3$, then $G$ is not isomorphic to any graph in Figure \ref{Figure: list of forbidden graphs}.
	\end{claim}

	\begin{claim}\label{Claim5: 2K2,claw-free characterisation}
		If $G$ is $2K_{2}$-free, then $G/e$ is $2K_{2}$-free for any $e \in E(G)$.
	\end{claim}

	\begin{claim}\label{Claim6: 2K2,claw-free characterisation}
		If $G/e$ has a vertex set, $C$, that induces $C_{4}$, then $G$ has an induced $C_{4}$ or $C_{5}$. Further, if $G$ is $C_{5}$-free, then $G$ has either  vertex $v$ that is adjacent to exactly one vertex, two adjacent vertices, or three vertices in $C$ or two adjacent vertices, $u$ and $v \in V \setminus C$.
	\end{claim}

	\begin{claim}\label{Claim7: 2K2,claw-free characterisation}
		If $G$ is ($2K_{2}$, claw)-free with $\alpha(G) \geq 3$, then $G/e$ is $C_{5}$-free for any $e \in E(G)$.
	\end{claim}
	\begin{proof}
		For any edge, $e \in E(G)$, if $G/e$ has an induced $C_{5}$, then $G$ has an induced $C_{5}$ or $C_{6}$. According to Theorem \ref{Theorem: brause}, $G$ does not have a vertex set that induces $C_{5}$. Moreover, because $G$ is $2K_{2}$-free, $G$ has no vertex set that induces $C_{6}$. 
	\end{proof}

Let $C=\{p,q,r,s\} \subseteq V$ that induces $C_{4}$ in $G$ with edges $pq,qr,rs$, and $sp$ .
\begin{claim}\label{Claim8: 2K2,claw-free characterisation}
	$C$ is dominating.
\end{claim}
\begin{proof}
	For the sake of contradiction, we assume that $C$ does not dominate. Therefore, there is a vertex, $v \notin C$, that is not adjacent to any vertex in $C$. Because $G$ is connected, $v$ is adjacent to vertex $u \notin C$. If $u$ is adjacent to no more than one vertex in $C$, say $p$, or exactly two adjacent vertices in $C$, say $p,q$, then $\{r,s,u,v\}$ induces $2K_{2}$ in $G$, and this is a contradiction. Otherwise, $u$ is adjacent to exactly either two nonadjacent vertices in $C$, say $p,r$, or at least three vertices in $C$, say $p,q,r$, and $\{p,r,u,v\}$ induces a claw in $G$, which is a contradiction.
\end{proof}

\begin{claim}\label{Claim9: 2K2,claw-free characterisation}
	If $v \notin C$, such that $N_{C}(v)=1$ or $v$ is adjacent to exactly two nonadjacent vertices in $C$, then $G$ is not claw-free.
\end{claim}
\begin{proof}
	We assume that $v$ is adjacent to exactly either a vertex in $C$, say $p$, or two nonadjacent vertices in $C$, say $p$ and $r$. Hence, $\{p,q,s,v\}$ induces a claw in $G$.
\end{proof}

\begin{claim}\label{Claim10: 2K2,claw-free characterisation}
	If $u$ and $v \notin C$, where $|N_{C}(u)|=|N_{C}(v)|=2$, then $u$ and $v$ are adjacent.
\end{claim}
\begin{proof}
	For the sake of contradiction, we assume that there are two vertices, $u$ and $v \notin C$, where $|N_{C}(u)|=|N_{C}(v)|=2$ but $u$ and $v$ are not adjacent. According to Claim \ref{Claim9: 2K2,claw-free characterisation}, neither $u$ nor $v$ is adjacent to two nonadjacent vertices in $C$. Thus, let $u$ be adjacent to $p$ and $q$. If $v$ is adjacent to $p$ and $q$, then $\{p,s,u,v\}$ induces a claw in $G$, which is a contradiction. Otherwise, $v$ is adjacent to $r$ or $s$; then, $\{p,r,u,v\}$ or $\{q,s,u,v\}$ induces $2K_{2}$ in $G$, which is a contradiction. 
\end{proof}

\begin{claim}\label{Claim11: 2K2,claw-free characterisation}
	If $S$ is an independent set in $G$ with $|S| \geq 3$, then $S \cap C$ is empty.
\end{claim}
\begin{proof}
	For the sake of contradiction, we assume that there is an independent set, $S$, in $G$ with $|S| \geq 3$ and $S \cap C$ is not empty. If $|S \cap C|=2$, then there is a vertex in $S$ that is adjacent to exactly one or two nonadjacent vertices in $C$, which contradicts Claims \ref{Claim8: 2K2,claw-free characterisation} and \ref{Claim9: 2K2,claw-free characterisation}. Otherwise, $|S \cap C|=1$. Let $|S \cap C|=\{p\}$ and $u,v \in S \setminus \{p\}$. According to Claims \ref{Claim8: 2K2,claw-free characterisation} and \ref{Claim10: 2K2,claw-free characterisation}, at most one of $u$ and $v$ is adjacent to exactly two adjacent vertices in $C$. If $u$ and $v$ are adjacent to $q$ or $s$, then $\{p,q,u,v \}$ (or $\{p,s,u,v \}$) induces a claw in $G$, which is a contradiction.
\end{proof}

\begin{claim}\label{Claim12: 2K2,claw-free characterisation}
	If $v \notin C$ and $|N_{C}(v)|=2$, then $v \notin S$, where $S$ is an independent set in $G$ and $|S| \geq 3$.
\end{claim}
\begin{proof}
	According to Claim \ref{Claim9: 2K2,claw-free characterisation}, $v$ is not adjacent to two nonadjacent vertices in $C$. Thus, for the sake of contradiction, we assume that there is a vertex, $v \notin C$, that is adjacent to exactly two adjacent vertices in $C$, say $p$ and $q$, and $v \in S$, where $S$ is an independent set in $G$ with $|S| \geq 3$. According to Claim \ref{Claim11: 2K2,claw-free characterisation}, $S$ does not have any vertex from $C$. Let $u$, $w \in S \setminus \{v\}$. According to Claims \ref{Claim8: 2K2,claw-free characterisation} and \ref{Claim10: 2K2,claw-free characterisation} and because $|N_{C}(v)|=2$, $N_{C}(u)$ and $N_{C}(w)$ are at least equal to 3. If $u$ and $w$ are adjacent to $p$ (or $q$), then $\{ p,u,v,w\}$ (or $\{ q,u,v,w\}$) induces a claw in $G$, which is a contradiction. Otherwise, w.l.o.g. $u$ is adjacent to $p,r$, and $s$, and $w$ is adjacent to $q,r$, and $s$. Therefore, $\{ p,r,v,w\}$ induces $2K_{2}$ in $G$, which is a contradiction.
\end{proof}

\begin{claim}\label{Claim13: 2K2,claw-free characterisation}
	If $v \notin C$ and $|N_{C}(v)|=3$, then $v \notin S$, where $S$ is an independent set in $G$ and $|S| \geq 3$.
\end{claim}
\begin{proof}
	For the sake of contradiction, we assume that there is a vertex, $v \notin C$, with $|N_{C}(v)|=3$ and $v \in S$, where $S$ is an independent set in $G$ and $|S| \geq 3$. According to Claims \ref{Claim8: 2K2,claw-free characterisation}, \ref{Claim11: 2K2,claw-free characterisation}, and \ref{Claim12: 2K2,claw-free characterisation}, if $u \in S$, then $|N_{C}(u)| \geq 3$. Let $u,w \in S \setminus {v}$. According to the pigeonhole principle, $|N_{C}(u) \cap N_{C}(v) \cap N_{C}(w)|  \geq 1$. Let $p \in N_{C}(u) \cap N_{C}(v) \cap N_{C}(w)$; then, $\{p,u,v,w\}$ induces a claw in $G$, which is a contradiction.
\end{proof}

\begin{claim}\label{Claim14: 2K2,claw-free characterisation}
	If $v \notin C$ and $|N_{C}(v)|=4$, then $v \notin S$, where $S$ is an independent set in $G$ and $|S| \geq 3$.
\end{claim}
\begin{proof}
	For the sake of contradiction, we assume that there is a vertex, $v \notin C$, with $|N_{C}(v)|=4$ and $v \in S$ where $S$ is an independent set in $G$ and $|S| \geq 3$. According to Claims \ref{Claim8: 2K2,claw-free characterisation}, \ref{Claim11: 2K2,claw-free characterisation}, \ref{Claim12: 2K2,claw-free characterisation}, and \ref{Claim13: 2K2,claw-free characterisation}, if  $u \in S$, then $N_{C}(u)=4$. Then, any three vertices in $S$ and any vertex in $C$ form a vertex set of cardinality $4$ that induces a claw in $G$, which is a contradiction.
\end{proof}

	Based on Claims \ref{Claim8: 2K2,claw-free characterisation}, $\dots$, \ref{Claim14: 2K2,claw-free characterisation}, we obtain the following claim.
		\begin{claim}\label{Claim15: 2K2,claw-free characterisation}
		If $G$ is ($2K_{2}$,claw)-free with $\alpha(G) \geq 3$, then $G/e$ is $C_{4}$-free for any $ e \ in E (G) $.
	\end{claim}

\renewcommand{\qedsymbol}{$\square$}
Based on Claims \ref{Claim4: 2K2,claw-free characterisation}, \ref{Claim5: 2K2,claw-free characterisation}, \ref{Claim6: 2K2,claw-free characterisation}, \ref{Claim7: 2K2,claw-free characterisation}, and \ref{Claim15: 2K2,claw-free characterisation}, the proof is complete.
\end{proof}

\section{Unbalanced split graphs}\label{Section: unbalanced split}
A \emph{star} denoted by $S_{n}$ is a graph constructed by $E_{1} \vee E_{n}$, where $n$ is a nonpositive integer.
A split graph, $G$, is called \emph{balanced split} if $G$ has a $KS$-partition, where $|K|=\omega(G)$ and $|S|=\alpha(G)$, and \emph{unbalanced split}, otherwise. Based on the work in \cite{hammer1981splittance}, the following theorem appears in \cite{golumbic2004algorithmic} and its proof is presented in \cite{cheng2016split}:
\begin{theorem}[\cite{hammer1981splittance},\cite{golumbic2004algorithmic},\cite{cheng2016split}]\label{unique balanced partition}
	For any $KS$-partition of split graph $G$, exactly one of the following holds:
	(i) $|K|=\omega(G)$ and $|S|=\alpha(G)$.\\
	(ii) $|K|=\omega(G)-1$ and $|S|=\alpha(G)$.\\
	(iii) $|K|=\omega(G)$ and $|S|=\alpha(G)-1$.\\
	Moreover, in $(i)$, the $KS$-partition is unique.
\end{theorem}

\begin{theorem}\label{Theorem: unbalanced split characterisation}
	Let $G$ be a split graph that is not isomorphic to any $S_{n}$ for $n \geq 2$. Then, $G$ is an unbalanced split if and only if there is an edge, $e \in E(G)$, such that $\omega(G/e) = \omega(G)-1$ and $G/e$ is an unbalanced split.
\end{theorem}
\begin{proof}
	Let $G$ be an unbalanced split graph, then $V(G)$ can be partitioned into a $KS$-partition, where $|K|=\omega(G)-1$ and $|S|=\alpha(G)$. For any two adjacent vertices, $u$ and $v \in K$, $V(G/uv)$ can be partitioned into $K^{'}S$-partitions, where $|K^{'}|=\omega(G)-2$ and $|S|=\alpha(G)$. Thus, $G/uv$ is an unbalanced split with $\omega(G/uv) = \omega(G)-1$.
	
	Conversely, we prove that if $G$ is a balanced split, then for any edge, $e$, $G/e$ is a balanced split or $\omega(G/e) = \omega(G)$. Because $G$ is a balanced split, then based on Theorem \ref{unique balanced partition}, $V(G)$ can be partitioned into the unique $KS$-partition, where $|K|=\omega(G)$ and $|S|=\alpha(G)$. For any two adjacent vertices, $u$ and $v \in K$, $V(G/uv)$ can be partitioned into $K^{'}S$-partitions, where$|K^{'}|=\omega(G)-1$ and $|S|=\alpha(G)$. Thus, either $G/uv$ is a balanced split with $|K^{'}|=\omega(G)-1$ or $G/uv$ is an unbalanced split with $|K^{'}|=\omega(G)$. For any two adjacent vertices, $u \in K$ and $v \in S$, $V(G/uv)$ can be partitioned into $KS^{'}$-partitions, where $|K|=\omega(G)$ and $|S^{'}|=\alpha(G)-1$. Thus, $\omega(G/uv) = \omega(G)$.
\end{proof}

\subsection{Pseudo-split graphs}\label{Section: pseudo-split}
A graph is called $G$ \emph{pseudo-split} if $G$ is ($2K_{2}$, $C_{4}$)-free. In \cite{blazsik1993graphs}, the family of ($2K_{2}$,  $C_{4}$)-free graphs was investigated and later referred to as pseudo-split graphs in \cite{maffray1994linear}. Different authors have characterised pseudo-split graphs as follows. 
\begin{theorem} [\cite{blazsik1993graphs} , \cite{maffray1994linear}] \label{Theorem: pseudo-split}
	A graph is a pseudo-split if and only if its vertex set can be partitioned into
	three sets, $A, B,$ and $C$, such that $A$ induces a clique, $B$ induces an independent set, and $C$ induces $C_{5}$ or is
	empty, such that there are all possible edges between $A$ and $C$, and there are no edges between $B$ and $C$.
\end{theorem}

The following result is obtained from Theorems \ref{Split graph characterisation} and \ref{Theorem: pseudo-split}:
\begin{theorem}\label{Pseudo-split graph characterisation 2}
	Given a connected graph, $G$, that is not isomorphic to any graph in Figure \ref{Figure: list of forbidden graphs}, $G$ is pseudo-split if and only if
	\begin{itemize}
		\item $G/e$ is split for any $e \in E(G)$ or
		\item $V(G)$ can be partitioned into three sets, $A, B,$ and $C$, such that $A$ induces a clique, $B$ induces an independent set, and $C$ induces $C_{5}$, where there are all possible edges between $A$ and $C$, and there are no edges between $B$ and $C$.
	\end{itemize}
\end{theorem}

\subsection{Nordhaus--Gaddum graphs}
A graph $G$ is called \emph{Nordhaus--Gaddum} (\emph{$NG$} in short) if $\chi(G) +\chi(\bar{G})$ $=$ $|V(G)|$ $+$ $|V(\bar{G})| +1$. $NG$ graphs were investigated and characterised in \cite{finck1968chromatic}, \cite{starr2008complementary}, \cite{collins2012nordhaus}, and \cite{cheng2016split}.
Based on the work in \cite{collins2012nordhaus}, the authors in \cite{cheng2016split} proved the following characterisation for $NG$ graphs:
\begin{theorem}[\cite{cheng2016split}]
	Graph $G$ is an $NG$-graph if and only if $G$ is a pseudo-split but not a balanced split.
\end{theorem}
In other words, the aforementioned theorem can be formulated as follows:
\begin{theorem}[\cite{cheng2016split}]\label{last}
	Graph $G$ is an $NG$-graph if and only if
	\begin{itemize}
		\item $G$ is unbalanced split or
		\item $V(G)$ can be partitioned into three sets, $A, B,$ and $C$, such that $A$ induces a clique, $B$ induces an independent set, and $C$ induces $C_{5}$, where there are all possible edges between $A$ and $C$, and there are no edges between $B$ and $C$.
	\end{itemize}
\end{theorem}

Based on Theorems \ref{Theorem: unbalanced split characterisation} and \ref{last}, the following result is obtained:
\begin{theorem}
		Graph $G$ that is not isomorphic to any $S_{n}$ for $n \geq 2$ is an $NG$-graph if and only if 
	\begin{itemize}
		\item there is an edge, $e \in E(G)$, such that $\omega(G/e) = \omega(G)-1$ and $G/e$ is unbalanced split, or
		\item $V(G)$ can be partitioned into three sets, $A, B,$ and $C$, such that $A$ induces a clique, $B$ induces an independent set, and $C$ induces $C_{5}$, where there are all possible edges between $A$ and $C$, and there are no edges between $B$ and $C$.
	\end{itemize}
\end{theorem}

\section{Acknowledgments}
The author would like to sincerely thank Peter Tittmann for his  constructive suggestions and discussions.

\bibliographystyle{apalike}
\bibliography{mybib}

\end{document}